\documentclass[11pt]{article}

\usepackage{amsfonts}
\usepackage{amsxtra}
\usepackage{amsthm}
\usepackage{amssymb}
\usepackage{amsmath,amscd}
\usepackage{supertabular}
\usepackage{pstricks}
\usepackage{pst-plot}
\usepackage{pstricks-add}
\usepackage{pst-eps}
\usepackage{makeidx}
\usepackage{bm}

\newtheorem{theorem}{Theorem}[section]
\newtheorem{lemma}[theorem]{Lemma}

\newtheorem{remark}[theorem]{Remark}

\makeindex

\begin{document}

\title{Two Lorentzian Lattices}

\author
{Gert Heckman and Sander Rieken\\
Radboud University Nijmegen}

\date{\today}

\maketitle

\begin{abstract}
We prove a conjecture formulated in \cite{Heckman 2013}, which in turn
provides a good deal of evidence for the monstrous proposal of 
Daniel Allcock \cite{Allcock 2009}.
\end{abstract}

\section{Introduction}

Let $\mathbb{Z}^{n,1}$ be the odd unimodular Lorentzian lattice with basis
$e_0,e_1,\cdots,e_n$ and inner product
\[ (e_0,e_0)=-1,(e_0,e_p)=0,(e_p,e_q)=\delta_{pq} \]
for $p,q=1,\cdots,n$. Let $\mathbb{R}^{n,1}=\mathbb{R}\otimes\mathbb{Z}^{n,1}$ 
be the corresponding Lorentzian real vector space, and let 
\[ \mathbb{R}^{n,1}_+=\{x\in\mathbb{R}^{n,1};(x,x)<0,x_0>0\} \]
be the connected component of the light cone complement containing $e_0$
and let
\[ \mathbb{B}^n=\mathbb{R}^{n,1}_+/\mathbb{R}_+ \]
be the real hyperbolic ball of dimension $n$. The index $2$ subgroup of the 
Lorentz group $\mathrm{O}(\mathbb{R}^{n,1})$ preserving the component
$\mathbb{R}^{n,1}_+$ is called the forward Lorentz group and is denoted
$\mathrm{O}_+(\mathbb{R}^{n,1})$. Clearly $\mathrm{O}_+(\mathbb{R}^{n,1})$
has two connected components distinguished by the sign of the determinant.
It is well known that
\[ \Gamma^n=\mathrm{O}_+(\mathbb{Z}^{n,1})=
\mathrm{O}_+(\mathbb{R}^{n,1})\cap\mathrm{O}(\mathbb{Z}^{n,1}) \]
is a discrete subgroup of $\mathrm{O}_+(\mathbb{R}^{n,1})$ acting on
$\mathbb{B}^n$ properly discontinuously with cofinite volume. It contains 
reflections 
\[ s_{\alpha}(x)=x-2(x,\alpha)\alpha/\alpha^2 \]
in roots $\alpha\in\mathbb{Z}^{n,1}$ of norm $1$ or norm $2$. Here we denote
$\alpha^2=(\alpha,\alpha)$ for the norm of $\alpha\in\mathbb{Z}^{n,1}$.
The next theorem is due to Vinberg (for $n\leq17$ and for $n=18,19$ in 
collaboration with Kaplinskaya) \cite{Vinberg 1980}. 

\begin{theorem}\label{Vinberg theorem}
For $2\leq n\leq19$ the group $\Gamma^n=\mathrm{O}_+(\mathbb{Z}^{n,1})$ is
generated by refections in roots $\alpha\in\mathbb{Z}^{n,1}$ of norm $1$ or norm $2$.
Moreover the Coxeter diagram for $n=7$ is given by
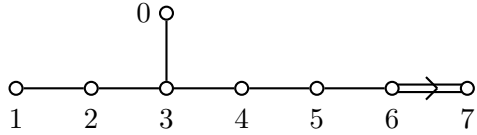
\begin{center}
\psset{unit=1mm}
\begin{pspicture}*(-50,-8)(30,15)
\pscircle(-20,10){1}
\pscircle(-40,0){1}
\pscircle(-30,0){1}
\pscircle(-20,0){1}
\pscircle(-10,0){1}
\pscircle(0,0){1}
\pscircle(10,0){1}
\pscircle(20,0){1}

\psline(-20,9)(-20,1)
\psline(-39,0)(-31,0)
\psline(-29,0)(-21,0)
\psline(-19,0)(-11,0)
\psline(-9,0)(-1,0)
\psline(1,0)(9,0)
\psline(10.7,0.5)(19.3,0.5)
\psline(10.7,-0.5)(19.3,-0.5)
\psline(14.5,1.5)(16,0)
\psline(14.5,-1.5)(16,0)

\rput(-23,10){$0$}
\rput(-40,-4){$1$}
\rput(-30,-4){$2$}
\rput(-20,-4){$3$}
\rput(-10,-4){$4$}
\rput(0,-4){$5$}
\rput(10,-4){$6$}
\rput(20,-4){$7$}
\end{pspicture}
\end{center}
with simple roots 
\[ \alpha_0=e_0-e_1-e_2-e_3,\alpha_1=e_1-e_2,\cdots,\alpha_6=e_6-e_7,\alpha_7=e_7 \]
and for $n=13$ the Coxeter diagram takes the form
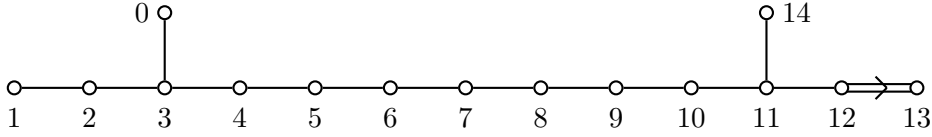
\begin{center}
\psset{unit=1mm}
\begin{pspicture}*(-42,-8)(82,15)
\pscircle(-20,10){1}
\pscircle(-40,0){1}
\pscircle(-30,0){1}
\pscircle(-20,0){1}
\pscircle(-10,0){1}
\pscircle(0,0){1}
\pscircle(10,0){1}
\pscircle(20,0){1}
\pscircle(30,0){1}
\pscircle(40,0){1}
\pscircle(50,0){1}
\pscircle(60,0){1}
\pscircle(70,0){1}
\pscircle(80,0){1}
\pscircle(60,10){1}

\psline(-20,9)(-20,1)
\psline(-39,0)(-31,0)
\psline(-29,0)(-21,0)
\psline(-19,0)(-11,0)
\psline(-9,0)(-1,0)
\psline(1,0)(9,0)
\psline(11,0)(19,0)
\psline(60,1)(60,9)
\psline(21,0)(29,0)
\psline(31,0)(39,0)
\psline(41,0)(49,0)
\psline(51,0)(59,0)
\psline(61,0)(69,0)
\psline(70.7,0.5)(79.3,0.5)
\psline(70.7,-0.5)(79.3,-0.5)
\psline(74.5,1.5)(76,0)
\psline(74.5,-1.5)(76,0)

\rput(-23,10){$0$}
\rput(-40,-4){$1$}
\rput(-30,-4){$2$}
\rput(-20,-4){$3$}
\rput(-10,-4){$4$}
\rput(0,-4){$5$}
\rput(10,-4){$6$}
\rput(20,-4){$7$}
\rput(30,-4){$8$}
\rput(40,-4){$9$}
\rput(50,-4){$10$}
\rput(60,-4){$11$}
\rput(70,-4){$12$}
\rput(80,-4){$13$}
\rput(64,10){$14$}
\end{pspicture}
\end{center}
with simple roots
\begin{eqnarray*}
&\alpha_0=e_0-e_1-e_2-e_3,\alpha_1=e_1-e_2,\cdots,\alpha_{11}=e_{11}-e_{12},\\
&\alpha_{12}=e_{12}-e_{13},\alpha_{13}=e_{13},\alpha_{14}=3e_0-e_1-\cdots-e_{11}. 
\end{eqnarray*}
\end{theorem}

Let $\Gamma^n_1$ be the normal subgroup of $\Gamma^n$ generated by the reflections
in norm $1$ roots. Clearly $\Gamma^n_1$ is a subgroup of the principal congruence
subgroup $\Gamma^n(2)$ of level $2$, and in fact we have equality if and only if 
$n\leq7$ as shown by Everitt, Ratcliffe and Tschantz 
\cite{Everitt--Ratcliffe--Tschantz 2010} 
(see \cite{Heckman--Rieken 2014a} for a quick proof). 
Let $D$ be the closed fundamental chamber in the closure of $\mathbb{R}^{n,1}_+$ for 
$\Gamma^n$ containing the point $e_0-(\epsilon_1e_1+\cdots+\epsilon_ne_n)$ 
in $\mathbb{R}^{n,1}_+$ (for $\epsilon_1>\cdots>\epsilon_n>0$ all small), in 
accordance with the choice of positive roots in case $n=7,13$ in the above theorem, 
and let $G\supset D$ be the closed fundamental chamber for the Coxeter group $\Gamma^n_1$. 
By Coxeter group theory
\[  G=\cup_{\gamma\in\Gamma^n_0}\;\gamma D\]
with $\Gamma^n_0$ the subgroup of $\Gamma^n$ generated by the simple roots 
of norm $2$, and $\Gamma^n=\Gamma^n_1\rtimes\Gamma^n_0$. The local structure 
of $x$ in $G$ near $x_0e_0$ with $x_0>0$ is given by $x_p\leq0$ for $p=1,\cdots,n$.

The Fano plano is the projective plane $\mathbb{P}^2(2)$ over a field of $2$
elements. It has $7$ points, denoted $\{a,a_i,c_i\}$ with $i=1,2,3$, and $7$ lines,
given by the line $z$ through $\{c_1,c_2,c_3\}$, the three lines $\{b_i\}$
through $\{a,a_i,c_i\}$ and the remaining three lines $\{d_i\}$ through
$\{c_i,a_j,a_k\}$ with $\{i,j,k\}=\{1,2,3\}$ in accordance with the picture
below on the left. The incidence graph $\mathrm{I}_{14}$ is bipartite with
$7$ black nodes from the set of points $\mathcal{P}$ and $7$ white nodes from 
the set of lines $\mathcal{L}$. In its picture below on the right an ordinary
bond means incidence, while the thick bond between $a_i$ and $d_i$ indicates that
$a_i$ and $d_j$ are incident (and so connected) if and only if $i\neq j$
\cite{Simons 2001}.
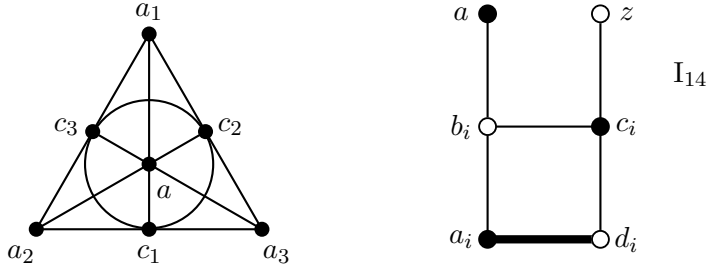
\begin{center}
\psset{unit=1mm}
\begin{pspicture}*(-50,-20)(50,20)

\psdot[dotsize=2](-30,-5)\psdot[dotsize=2](-30,12.32)
\psdot[dotsize=2](-45,-13.66)\psdot[dotsize=2](-15,-13.66)
\psdot[dotsize=2](-30,-13.66)\psdot[dotsize=2](-37.5,-0.67)
\psdot[dotsize=2](-22.5,-0.67)

\psline(-30,12.32)(-45,-13.66)(-15,-13.66)(-30,12.32)(-30,-13.66)
\psline(-37.5,-0.67)(-15,-13.66)\psline(-22.5,-0.67)(-45,-13.66)

\rput(-28,-8.7){$a$}\rput(-30,15.3){$a_1$}\rput(-47,-17){$a_2$}\rput(-13,-17){$a_3$}
\rput(-30,-17){$c_1$}\rput(-19.2,0){$c_2$}\rput(-41,0){$c_3$}

\pscircle(-30,-5){8.66}

\psline(15,15)(15,0)
\psline(15,0)(15,-15)
\psline(15,0)(30,0)
\psline(30,15)(30,0)
\psline(30,0)(30,-15)
\psline[linewidth=1.1](15,-15)(30,-15)

\psdot[dotsize=2.5](15,15)
\psdot[dotsize=2.5](15,-15)
\psdot[dotsize=2.5,dotstyle=Bo](15,0)
\psdot[dotsize=2.5](30,0)
\psdot[dotsize=2.5,dotstyle=Bo](30,-15)
\psdot[dotsize=2.5,dotstyle=Bo](30,15)

\rput(11.5,15){$a$}
\rput(33.5,15){$z$}
\rput(11.5,-15){$a_i$}
\rput(33.5,-15){$d_i$}
\rput(11.5,0){$b_i$}
\rput(33.5,0){$c_i$}

\rput(42,7){$\mathrm{I}_{14}$}

\end{pspicture}
\end{center}
The group of diagram automorphism $\mathrm{PGL}_3(2)\cdot2$ of the Coxeter diagram  
$\mathrm{I}_{14}$ has the simple index $2$ subgroup
$\mathrm{PGL}_3(2)\cong\mathrm{PSL_2(7)}$ of order $168$ as colour preserving
automorphisms. The remaining group elements are involutions corresponding to 
projective dualities and interchange black and white nodes. 

\begin{theorem}\label{projective plane of order 2 theorem}
Fix a bijection $\mathcal{P}\rightarrow\{1,\cdots,7\}$ and so $e_p\in\mathbb{Z}^{7,1}$
for $p\in\mathcal{P}$. Let us write $e_l=e_0-\sum_{p\in l}e_p\in\mathbb{Z}^{7,1}$ for 
$l\in\mathcal{L}$, and therefore $(e_p,e_q)=\delta_{pq}$ for all $p,q\in\mathcal{P}$,
$(e_l,e_m)=2\delta_{lm}$ for all $l,m\in\mathcal{L}$, and $(e_p,e_l)=-1,0$ if 
$p\in l,p\notin l$ respectively. If we denote
\[ P=\{x\in\mathbb{R}^{7,1};(x,e_p)\leq0,(x,e_l)\leq0\;
\mathrm{for\;all}\; p\in\mathcal{P},l\in\mathcal{L}\} \]
then the induced hyperbolic polytope in $\mathbb{B}^7$ has finite volume
and we have the inclusions $D\subset P\subset G$.
\end{theorem}

The proof of this result is a rather straightforward calculation given in the
next section. The interest of this theorem is its analogy with the theorem below,
which was conjectured in \cite{Heckman 2013} and provides a positive step
towards the monstrous proposal of Daniel Allcock \cite{Allcock 2009}.

Now, let us consider the projective plane $\mathbb{P}^2(3)$ over a field of $3$
elements, and denote as before by $\mathcal{P}$ and $\mathcal{L}$ the sets of its
$13$ points and $13$ lines respectively. Its incidence graph $\mathrm{I}_{26}$
with $26$ nodes indexed by $\mathcal{P}\sqcup\mathcal{L}$ has a picture given 
below with the same convention on the meaning of the thick bonds and the index
$i\in\{1,2,3\}$ as before \cite{Conway--Simons 2001}. Its automorphism group 
is the group $\mathrm{PGL}_3(3)\cdot2$ of order $11232=2^5\cdot3^3\cdot13$ consisting of 
projective transformations and projective dualities.
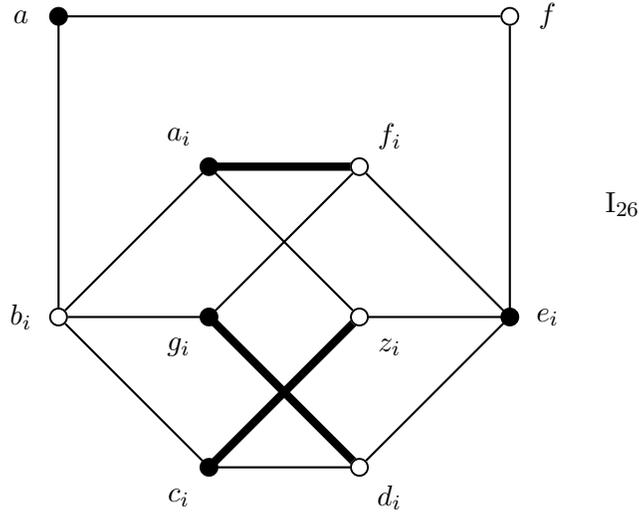
\begin{center}
\psset{unit=1mm}
\begin{pspicture}*(-40,-30)(50,45)

\psline(-30,0)(-30,40)
\psline(30,0)(30,40)
\psline(-30,40)(30,40)
\psline(-30,0)(-10,20)
\psline(-30,0)(-10,0)
\psline(-30,0)(-10,-20)
\psline(30,0)(10,20)
\psline(30,0)(10,0)
\psline(30,0)(10,-20)
\psline(-10,-20)(10,-20)
\psline[linewidth=1.1](-10,20)(10,20)
\psline[linewidth=1.1](-10,0)(10,-20)
\psline[linewidth=1.1](-10,-20)(10,0)
\psline(-10,20)(10,0)
\psline(-10,0)(10,20)

\psdot[dotsize=2.5,dotstyle=Bo](-30,0)
\psdot[dotsize=2.5](-10,0)
\psdot[dotsize=2.5,dotstyle=Bo](10,0)
\psdot[dotsize=2.5](30,0)
\psdot[dotsize=2.5](-10,20)
\psdot[dotsize=2.5,dotstyle=Bo](10,20)
\psdot[dotsize=2.5](-10,-20)
\psdot[dotsize=2.5,dotstyle=Bo](10,-20)
\psdot[dotsize=2.5](-30,40)
\psdot[dotsize=2.5,dotstyle=Bo](30,40)

\rput(-35,40){$a$}
\rput(35,40){$f$}
\rput(-35,0){$b_i$}
\rput(35,0){$e_i$}
\rput(-14,24){$a_i$}
\rput(14,24){$f_i$}
\rput(-14,-24){$c_i$}
\rput(14,-24){$d_i$}
\rput(-14,-4){$g_i$}
\rput(14,-4){$z_i$}
\rput(45,15){$\mathrm{I}_{26}$}

\end{pspicture}
\end{center}
The next result is the exact analogue of the previous theorem for the projective
plane of order $3$ rather than $2$.

\begin{theorem}\label{projective plane of order 3 theorem}
Fix a bijection $\mathcal{P}\rightarrow\{1,\cdots,13\}$ and write 
$e_l=e_0-\sum_{p\in l}e_p$ for $l\in\mathcal{L}$, and therefore
$(e_p,e_q)=\delta_{pq}$ for all $p,q\in\mathcal{P}$, $(e_l,e_m)=3\delta_{lm}$ 
for all $l,m\in\mathcal{L}$, and $(e_p,e_l)=-1,0$ if $p\in l,p\notin l$ respectively.
If we denote
\[ P=\{x\in\mathbb{R}^{13,1};(x,e_p)\leq0,(x,e_l)\leq0\;
\mathrm{for\;all}\; p\in\mathcal{P},l\in\mathcal{L}\} \]
then the induced hyperbolic polytope in $\mathbb{B}^{13}$ has finite volume
and we have the inclusion $P\subset G$.
\end{theorem}

Again the proof is by straightforward but more extensive (however not unpleasant) calculations.
The inclusion $P\subset G$ was conjectured in \cite{Heckman 2013} and in the final section 
of this paper we discuss its relevance towards the monstrous proposal of Daniel Allcock 
\cite{Allcock 2009}. 

\section{Proof of Theorem{\;\ref{projective plane of order 2 theorem}}}

In this section we consider the case $n=7$ and $\mathrm{I}_{14}$ is the incidence
graph of $\mathbb{P}^2(2)$ with black nodes from the set of points $\mathcal{P}$
and white nodes from the set of lines $\mathcal{L}$. The extremals of the fundamental 
chamber $D$ for the Coxeter group $\Gamma^7$ are spanned over $\mathbb{R}_+$ by 
\[ v_0=e_0,v_1=e_0-e_1,v_2=2e_0-e_1-e_2,v_p=3e_0-e_1-\cdots-e_p \]
for $p=3,\cdots,7$ as antidual basis of the basis of simple roots in
Theorem{\;\ref{Vinberg theorem}}. 

The group $\Gamma^7_0$ is the Weyl group $W(\mathrm{E}_7)$ as stabilizer of the vector 
$v_7$. Note that the fundamental chamber $G$ for the Coxeter group $\Gamma^7_1$
is bounded by $56$ (being $|W(\mathrm{E}_7)|/|W(\mathrm{E}_6)|$) 
walls corresponding to the positive norm $1$ roots
\[ e_p,e_0-e_p-e_q,2e_0-e_1-\cdots-e_7+e_p+e_q,3e_0-e_1-\cdots-e_7-e_p \]
for $p,q\in\mathcal{P}$ with $p\neq q$, making all together $7+\binom{7}{2}+\binom{7}{2}+7=56$ 
roots as should. This description is well known from the description of the
$56$ lines on a degree two del Pezzo surface \cite{Dolgachev--Ortland}. 
All these simple roots have the same inner product with the
vector $v_7$, which for that reason is called a Weyl vector for the chamber $G$.

The convex cone 
\[ P=\{x\in\mathbb{R}^{7,1};(x,e_p)\leq0,(x,e_l)\leq0\;
\mathrm{for\;all}\; p\in\mathcal{P},l\in\mathcal{L}\} \] has dihedral angles 
$\pi/4$ or $\pi/2$ and the reflections in the walls of $P$ generate a Coxeter
group with Coxeter diagram $\mathrm{I}_{14}$, but with all edges marked with $4$.
Its connected parabolic subdiagrams are of type $\mathrm{A}_3$, and
each such is contained in a parabolic subdiagram of type $3{\mathrm{A}}_3$
(in our picture of the Coxeter diagram $\mathrm{I}_{14}$ we left out the marks
$4$ on the bonds, and here we continue this convention, so $\mathrm{A}_3$
actually stands for $\tilde{\mathrm{C}}_2$ and this is indeed parabolic).
Since the Coxeter diagram has no Lann\'{e}r subdiagrams we conclude that
$P$ has finite hyperbolic volume by the Vinberg criterion \cite{Heckman CG 2013}.
Since the inclusion $D\subset P$ is trivial by looking around $v_{\mathcal{P}}$
it remains to show that $P\subset G$. Denote 
\[ v_{\mathcal{P}}=v_0,v_{\mathcal{L}}=v_7/\sqrt{2} \]
for two norm $-1$ vectors in $\mathbb{R}^{7,1}_+$, given as the intersection
of the mirrors perpendicular to the norm $1$ roots $e_p$ for $p\in\mathcal{P}$
and the intersection of the mirrors for the norm $2$ roots $e_l=e_0-\sum_{p\in l}e_p$
for $l\in\mathcal{L}$ respectively. 

\begin{lemma}\label{vertices of P over 2 lemma}
The actual vertices of the hyperbolic polytope obtained from $P$ are represented by
the $2$ vectors $v_{\mathcal{P}},v_{\mathcal{L}}$ together with 
the $56$ vectors
\[ v_{p,l}=2e_0-\sum_{q\notin l\sqcup\{p\},}e_q,
   v_{l,p}=(3e_0-2e_p-\sum_{q\in l}e_q)/\sqrt{2}  \]
for all $p\in\mathcal{P},l\in\mathcal{L}$ with $p\notin l$.   
The ideal vertices are represented by the $14$ vectors
\[ u_p=e_0-e_p,u_l=(2e_0-\sum_{p\notin l}e_p)/\sqrt{2} \]
for all $p\in\mathcal{P},l\in\mathcal{L}$.
\end{lemma}

\begin{proof}
Since the hyperbolic polytope obtained from $P$ is a finite volume convex
acute angled polytope its actual vertices are given by the elliptic subdiagrams
of its Coxeter diagram of maximal rank $7$. A connected elliptic subdiagram
of the diagram $\mathrm{I}_{14}$ with all edges marked with the number $4$
(however, in all figures above and below this mark $4$ will be deleted)
is of type $\mathrm{A}_1$ or $\mathrm{A}_2$. For example, if you leave out
from $\mathrm{I}_{14}$ the black node $a$ then the maximal subdiagram obtained
by also deleting the white nodes $b_i$ connected with $a$ has the left form
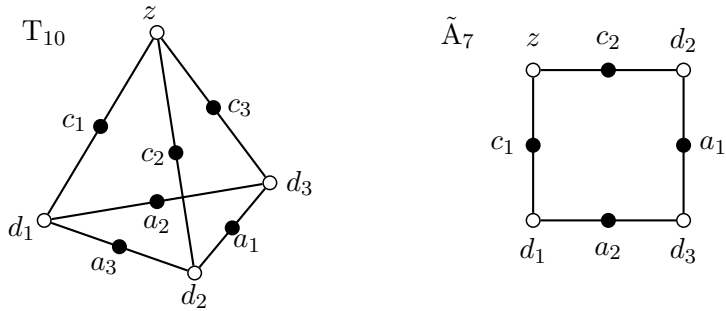
\begin{center}
\psset{unit=1mm}
\begin{pspicture}*(-70,-20)(50,30)

\psline(-55,-5)(-40,20)
\psline(-55,-5)(-35,-12)
\psline(-55,-5)(-25,0)
\psline(-40,20)(-35,-12)
\psline(-40,20)(-25,0)
\psline(-35,-12)(-25,0)

\psdot[dotsize=2,dotstyle=Bo](-55,-5)\psdot[dotsize=2,dotstyle=Bo](-40,20)
\psdot[dotsize=2,dotstyle=Bo](-35,-12)\psdot[dotsize=2,dotstyle=Bo](-25,0)

\psdot[dotsize=2](-47.5,7.5)\psdot[dotsize=2](-45,-8.5)
\psdot[dotsize=2](-40,-2.5)\psdot[dotsize=2](-37.5,4)
\psdot[dotsize=2](-32.5,10)\psdot[dotsize=2](-30,-6)

\rput(-41,22.6){$z$}\rput(-58,-6){$d_1$}\rput(-35,-15){$d_2$}\rput(-21,0){$d_3$}
\rput(-51,8){$c_1$}\rput(-40.5,3.5){$c_2$}\rput(-29,10){$c_3$}
\rput(-28,-8){$a_1$}\rput(-40,-5.5){$a_2$}\rput(-47,-11){$a_3$}

\rput(-55,20){$\mathrm{T}_{10}$}
\rput(0,20){$\tilde{\mathrm{A}}_7$}

\psline(10,15)(30,15)\psline(30,15)(30,-5)
\psline(10,-5)(30,-5)\psline(10,-5)(10,15)

\psdot[dotsize=2,dotstyle=Bo](10,15)\psdot[dotsize=2,dotstyle=Bo](30,15)
\psdot[dotsize=2,dotstyle=Bo](10,-5)\psdot[dotsize=2,dotstyle=Bo](30,-5)

\psdot[dotsize=2](10,5)\psdot[dotsize=2](30,5)
\psdot[dotsize=2](20,15)\psdot[dotsize=2](20,-5)

\rput(10,19){$z$}\rput(10,-9){$d_1$}
\rput(30,19){$d_2$}\rput(30,-9){$d_3$}
\rput(6,5){$c_1$}\rput(20,19){$c_2$}
\rput(34,5){$a_1$}\rput(20,-9){$a_2$}

\end{pspicture}
\end{center} 
and so is a tetrahedron with white nodes at the $4$ vertices and black nodes
at the midpoints of $6$ egdes. We denote this Coxeter diagram by $T_{10}$ and 
by straightforward inspection it has only one elliptic subdiagram of rank $6$, 
namely of type $6\mathrm{A}_1$ and consisting of the $6$ black nodes. 
Hence we recover the subdiagram of type $7\mathrm{A}_1$ in $\mathrm{I}_{14}$ 
consisting of the $7$ black nodes corresponding to the vertex $v_{\mathcal{P}}$ 
of $P$. Similarly the subdiagram of type $7\mathrm{A}_1$ in $\mathrm{I}_{14}$ 
consisting of the $7$ white nodes corresponds to the vertex $v_{\mathcal{L}}$. 

The other possibility is that we start with a subdiagram of type $\mathrm{A}_2$ 
of $\mathrm{I}_{14}$, for example the subdiagram with nodes $\{a,b_3\}$. 
Leaving out the connecting nodes $\{b_1,b_2,a_3,c_3\}$
gives us a diagram of type $\tilde{\mathrm{A}}_7$ (also called a free octagon)
and drawn on the right. Any elliptic subdiagram of rank $5$ of this free octagon 
is of type $\mathrm{A}_1\sqcup2\mathrm{A}_2$, and so we get an elliptic subdiagram
of $\mathrm{I}_{14}$ of type $\mathrm{A}_1\sqcup3\mathrm{A}_2$. These give the 
$28$ vertices $v_{p,l}$ as given above with $p\in\mathcal{P},l\in\mathcal{L},p\notin l$. 
A projective duality interchanges $e_0=v_{\mathcal{P}}$ with 
$v_{\mathcal{L}}=v_7/\sqrt{2}$, and the $e_p$ with the vectors
$e_l/\sqrt{2}=(e_0-\sum_{q\in l}e_q)/\sqrt{2}$ for $p\in\mathcal{P},l\in\mathcal{L}$. 
By direct calculation the vertices $v_{p,l}$ give by projective duality the
$28$ vertices $v_{l,p}$ as given above.

The ideal vertices of $P$ correspond to the maximal parabolic subdiagrams of type 
$3\mathrm{A}_3$ of rank $6$, and the $14$ vertices $u_p$ and $u_l$ for
$p\in\mathcal{P},l\in\mathcal{L}$ follow by a direct computation.
\end{proof}

\begin{remark}
The vector $v_{\mathcal{P}}=v_0=e_0$ is a vertex of all three hyperbolic polytopes
associated with $D,P,G$ and the local structure near this vertex of $P$ and $G$
coincides and is equal to $\cup_{\gamma}\gamma D$ by letting $\gamma$ run over the 
symmetric group $S_7=W(\mathrm{A}_6)$ as stabilizer of the edge of $D$ from $v_0$ to $v_7$.
There are $7$ edges of $P$ and $G$ from $v_0$ to the $7$ ideal vertices $u_p=e_0-e_p$.
There are $\binom{7}{2}$ faces of dimension $2$ of $P$ and $G$ containing $v_0$,
which are right angled triangles at $v_0$ with $2$ ideal vertices $u_p,u_q$ for 
$p,q\in\mathcal{P}$ distinct. There are $\binom{7}{3}$ faces of dimension $3$ of $G$ 
containing $v_0$, which are double tetrahedra with $3$ ideal vertices $u_p,u_q,u_r$ 
for $p,q,r\in\mathcal{P}$ distinct and just one more actual vertex $2e_0-e_p-e_q-e_r$
obtained from $v_0$ by reflection in the norm 2 root $e_0-e_p-e_q-e_r$
\cite{Heckman--Rieken 2014a}. Apparently only those for which $p,q,r$ are not
collinear give $7\cdot6\cdot4/3!=28$ faces of dimension $3$ for $P$ leading to
the vertices $v_{p,l}$ as in the lemma. The remaining vertices of $P$ are
obtained by projective duality.
\end{remark}

The proof of the inclusion $P\subset G$ as stated in 
Theorem{\;\ref{projective plane of order 2 theorem}} follows now by direct inspection. The 
vertices $v_{\mathcal{P}}$, $v_{p,l}$ and $u_p$ for $p\in\mathcal{P},l\in\mathcal{L},p\notin l$
are also vertices of $G$. By direct inspection the four types of simple roots 
\[ e_p,e_0-e_p-e_q,2e_0-e_1-\cdots-e_7+e_p+e_q,3e_0-e_1-\cdots-e_7-e_p \]
for the Gosset chamber $G$ take values from 
\[ \{-2,-1,0\},\{-3,-2,-1,0\},\{-4,-3,-2,-1\},\{-4,-3,-2\} \]
on the actual vertices $\sqrt{2}v_{l,p}=3e_0-2e_p-\sum_{q\in l}e_q$ of $P$ respectively, 
and values from
\[ \{-1,0\},\{-2,-1,0\},\{-2,-1,0\},\{-2,-1\} \]
on the ideal vertices $\sqrt{2}u_l$ of $P$ respectively.
Finally these simple roots take the single value $-1$ on the Weyl vector $v_{\mathcal{L}}$
of $G$. Since all these numbers are $\leq0$ we conclude that all the vertices of $P$ are 
contained in $G$, which by the Minkowski convex hull theorem implies $P\subset G$.

The hyperbolic polytope $G$ has $576=2^6\cdot3^2$ (being $|W(\mathrm{E}_7)|/|W(\mathrm{A}_6)|$)
actual vertices, which equals $2^7(1+28/8)$ as should. Indeed, the hyperbolic polytope
$G$ has a tiling by $2^7$ congruent copies of $P$ meeting all at $v_{\mathcal{L}}$, while
just of them one contains $v_{\mathcal{P}}$, and $2^3$ of them meet at each $v_{p,l}$.

\section{An odd presentation for $W(\mathrm{E}_7)$}

The Weyl group $W(\mathrm{E}_6)$ has a remarkable presentation due to Christopher Simons
\cite{Simons 2005} as factor group of the hyperbolic Coxeter group $W({\mathrm{P}_{10}})$ 
modulo deflation of the free hexagons. Here $\mathrm{P}_{10}$ is our notation for
the Petersen graph. This presentation of $W(\mathrm{E}_6)$ was given a geometric explanation 
\cite{Heckman--Rieken 2014a} using the Allcock--Carlson--Toledo period map for the 
moduli space of cubic surfaces \cite{Allcock--Carlson--Toledo 2002} and its reality 
analysis by Yoshida \cite{Yoshida 2001}. In fact, Yoshida only discussed the maximal real
component of real cubic surfaces with $27$ real lines, which is also the component we used.
The complete reality analysis for the $5$ components with $27,15,7,3,3$ real lines
respectively was worked out in \cite{Allcock--Carlson--Toledo 2010}. 

In this section we tell a similar story for $W(\mathrm{E}_7)$ using the Kondo period
map for the moduli space of quartic curves \cite{Kondo 2000} and its reality analysis by
Heckman and Rieken \cite{Rieken 2015},\cite{Heckman--Rieken 2015}. Again only the
maximal real component of real quartic curves with $4$ components or equivalently with
$28$ real bitangents is needed for this purpose.

\begin{theorem}
If $\mathrm{T}_{10}$ is the tetrahedral Coxeter diagram 
(but this time bonds do have mark $3$ rather than $4$, and so are deleted as usual), 
then the Weyl group $W(\mathrm{E}_7)$ has a presentation as factor group of the 
hyperbolic Coxeter group $W(\mathrm{T}_{10})$ modulo deflation of the free octagons.
\end{theorem}

\begin{proof}
Consider the Coxeter diagram
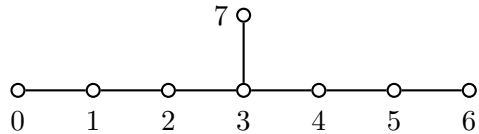
\begin{center}
\psset{unit=1mm}
\begin{pspicture}*(-60,-7)(20,13)
\pscircle(-20,10){1}
\pscircle(-50,0){1}
\pscircle(-40,0){1}
\pscircle(-30,0){1}
\pscircle(-20,0){1}
\pscircle(-10,0){1}
\pscircle(0,0){1}
\pscircle(10,0){1}

\psline(-20,9)(-20,1)
\psline(-49,0)(-41,0)
\psline(-39,0)(-31,0)
\psline(-29,0)(-21,0)
\psline(-19,0)(-11,0)
\psline(-9,0)(-1,0)
\psline(1,0)(9,0)

\rput(-23,10){$7$}
\rput(-50,-4){$0$}
\rput(-40,-4){$1$}
\rput(-30,-4){$2$}
\rput(-20,-4){$3$}
\rput(-10,-4){$4$}
\rput(0,-4){$5$}
\rput(10,-4){$6$}
\end{pspicture}
\end{center}
of type $\tilde{\mathrm{E}}_7$ as subdiagram of $\mathrm{T}_{10}$ via the numeration
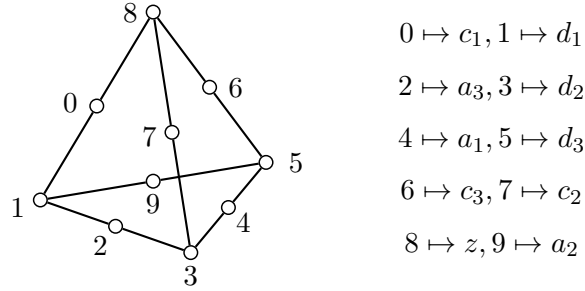
\begin{center}
\psset{unit=1mm}
\begin{pspicture}*(-70,-17 )(30,22)

\psline(-55,-5)(-40,20)
\psline(-55,-5)(-35,-12)
\psline(-55,-5)(-25,0)
\psline(-40,20)(-35,-12)
\psline(-40,20)(-25,0)
\psline(-35,-12)(-25,0)

\psdot[dotsize=2,dotstyle=Bo](-55,-5)\psdot[dotsize=2,dotstyle=Bo](-40,20)
\psdot[dotsize=2,dotstyle=Bo](-35,-12)\psdot[dotsize=2,dotstyle=Bo](-25,0)

\psdot[dotsize=2,dotstyle=Bo](-47.5,7.5)\psdot[dotsize=2,dotstyle=Bo](-45,-8.5)
\psdot[dotsize=2,dotstyle=Bo](-40,-2.5)\psdot[dotsize=2,dotstyle=Bo](-37.5,4)
\psdot[dotsize=2,dotstyle=Bo](-32.5,10)\psdot[dotsize=2,dotstyle=Bo](-30,-6)

\rput(-43,20){$8$}\rput(-58,-6){$1$}\rput(-35,-15){$3$}\rput(-21,0){$5$}
\rput(-51,8){$0$}\rput(-40.5,3.5){$7$}\rput(-29,10){$6$}
\rput(-28,-8){$4$}\rput(-40,-5.5){$9$}\rput(-47,-11){$2$}

\rput(5,17){$0\mapsto c_1,1\mapsto d_1$}
\rput(5,10){$2\mapsto a_3,3\mapsto d_2$}
\rput(5,3){$4\mapsto a_1,5\mapsto d_3$}
\rput(5,-4){$6\mapsto c_3,7\mapsto c_2$}
\rput(5,-11){$8\mapsto z,9\mapsto a_2$}
\end{pspicture}
\end{center} 
Let $\alpha_1,\cdots,\alpha_7$ be the simple roots in $R_+(\mathrm{E}_7)$ and put 
\begin{eqnarray*}
\alpha_0 &=& -(2\alpha_1+3\alpha_2+4\alpha_3+3\alpha_4+2\alpha_5+\alpha_6+2\alpha_7) \\
\alpha_8   &=& \alpha_1+2\alpha_2+3\alpha_3+2\alpha_4+\alpha_5+2\alpha_7 \\
\alpha_9 &=& \alpha_2+2\alpha_3+2\alpha_4+2\alpha_5+\alpha_6+\alpha_7 
\end{eqnarray*}
as vectors in $R(\mathrm{E}_7)$. Then we have $(\alpha_x,\alpha_y)=0$ for any two 
disconnected nodes $x,y$ of $\mathrm{T}_{10}$, while $(\alpha_x,\alpha_y)=-1$  for any 
two connected nodes $x,y$ of $\mathrm{T}_{10}$, except for $\{x,y\}=\{7,8\},\{5,9\}$ 
in which case this value is $1$. Hence the group $W(\mathrm{T}_{10})$ modulo deflation 
of the free octagons admits an epimorphsm onto $W(\mathrm{E}_7)$. 

There are three free hexagons, obtained by leaving out the midpoints 
$\{0,4\}$,$\{2,6\}$,$\{7,9\}$ of pairs of opposite edges, and their deflations amount
to the relations
\[ s_1s_2s_3s_7s_8s_6s_5s_9=1,s_1s_9s_5s_4s_3s_7s_8s_0=1,s_1s_2s_3s_4s_5s_6s_8s_0=1 \]
respectively. We claim that the generators $s_8,s_9,s_0$ are superfluous.
Indeed $s_9$ is a word in $s_1,\cdots,s_7$ using the second and the third relation.
By the first relation this implies that $s_8$ is also a word in $s_1,\cdots,s_7$.
Finally by the second relation we conclude that $s_0$ is also a word in $s_1,\cdots,s_7$.
This proves that the Coxeter group $W(\mathrm{T}_{10})$ modulo deflation of the 
free octagons is indeed isomorphic to $W(\mathrm{E}_7)$.
\end{proof}

The Coxeter diagram $\mathrm{T}_{10}$  has the group $S_4$ as group of diagram
automorphisms, and the deflation relations are also preserved. In turns out that this
$S_4$ is a subgroup of $W(\mathrm{E}_7)$ acting by inner automorphism on the generators.
The root system $R(\mathrm{E}_7)$ has a subroot system of type $6R(\mathrm{A}_1)$,
which is unique up to conjugation by $W(\mathrm{E}_7)$ and is normalized by the
simple group $\mathrm{PGL}_3(2)$ of order $168$. This realizes the semidirect product
$2^6\rtimes\mathrm{PGL}_3(2)$ as a maximal subgroup of $W(\mathrm{E}_7)$, and this
monomorphism is unique up to conjugation. The group $S_4$ is a maximal subgroup of
$\mathrm{PGL}_3(2)$, unique up to (inner and outer) automorphisms of $\mathrm{PGL}_3(2)$.
In turn this realizes $S_4$ as subgroup of $W(\mathrm{E}_7)$ in a up to conjugation almost 
unique way, the difference between the $2$ monomorphisms is coming from the center $\pm1$
of $W(\mathrm{E}_7)$. This monomorphism $S_4\hookrightarrow W(\mathrm{E}_7)$ acts 
on the odd presentation of $W(\mathrm{E}_7)$ as group of diagram automorphisms. 

Consider the extended tetrahedral Coxeter diagram $\mathrm{T}_{13}$ with $4$ white nodes
corresponding to norm $2$ vectors and $9$ black nodes corresponding to norm $1$ vectors in
$\mathbb{Z}^{6,1}$. The ordinary branches mean that the inner product is $-\sqrt{2}$ so the 
the dihedral angle between the $2$ mirrors is $\pi/4$, while the dashes branches mean that
the inner product is $-1$ so the $2$ mirrors are parallel. But for simplicity we leave out
the marks $4$ and $\infty$ from the branches as before. It is obtained from the tetrahedral
diagram $\mathrm{T}_{10}$ in the proof of Lemma{\;\ref{vertices of P over 2 lemma}} by adding
$3$ black nodes $\{b_i'\}$ dashed connected to the midpoints $\{a_i,c_i\}$ of pairs of
opposite edges of the tetrahedron, while the $3$ nodes $\{b_i'\}$ are also dashed connected
among each other. From this description it is clear that the Coxeter diagram 
$\mathrm{T}_{13}$ has $S_4$ as group of diagram automorphisms.

This diagram has parabolic subdiagrams of type $\tilde{\mathrm{A}}_1\sqcup2
\tilde{\mathrm{C}}_2$ of rank $5$ and no Lann\'{e}r subdiagrams. 
Hence the corresponding Coxeter group $W(\mathrm{T}_{13})$ has cofinite volume 
on hyperbolic space $\mathbb{B}^6$ by the Vinberg criterion. Its fundamental chamber 
$P'$ is just the wall of the fundamental chamber $P$ corresponding to the incidence 
Coxeter diagram $\mathrm{I}_{14}$ of the Fano plane as described in the introduction.
Indeed the roots with index $b_i'$ are obtained from those by $b_i$ in the Coxeter 
diagram $\mathrm{I}_{14}$ by orthogonal projection on the orthogonal complement of the root 
with index $a$. 

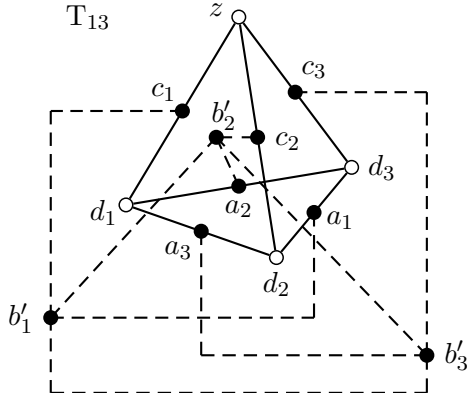
\begin{center}
\psset{unit=1mm}
\begin{pspicture}*(-80,-30)(0,25)

\psline(-55,-5)(-40,20)
\psline(-55,-5)(-35,-12)
\psline(-55,-5)(-25,0)
\psline(-40,20)(-35,-12)
\psline(-40,20)(-25,0)
\psline(-35,-12)(-25,0)

\psline[linestyle=dashed](-47.5,7.5)(-65,7.5)\psline[linestyle=dashed](-65,7.5)(-65,-20)
\psline[linestyle=dashed](-65,-20)(-30,-20)\psline[linestyle=dashed](-30,-20)(-30,-6)
\psline[linestyle=dashed](-45,-8.5)(-45,-25)\psline[linestyle=dashed](-45,-25)(-15,-25)
\psline[linestyle=dashed](-15,-25)(-15,10)\psline[linestyle=dashed](-15,10)(-32.5,10)
\psline[linestyle=dashed](-37.5,4)(-43,4)\psline[linestyle=dashed](-43,4)(-40,-2.5)

\psline[linestyle=dashed](-65,-20)(-65,-30)
\psline[linestyle=dashed](-65,-30)(-15,-30)
\psline[linestyle=dashed](-15,-30)(-15,-25)
\psline[linestyle=dashed](-15,-25)(-43,4)
\psline[linestyle=dashed](-65,-20)(-43,4)

\psdot[dotsize=2,dotstyle=Bo](-55,-5)\psdot[dotsize=2,dotstyle=Bo](-40,20)
\psdot[dotsize=2,dotstyle=Bo](-35,-12)\psdot[dotsize=2,dotstyle=Bo](-25,0)

\psdot[dotsize=2](-47.5,7.5)\psdot[dotsize=2](-45,-8.5)
\psdot[dotsize=2](-40,-2.5)\psdot[dotsize=2](-37.5,4)
\psdot[dotsize=2](-32.5,10)\psdot[dotsize=2](-30,-6)

\psdot[dotsize=2](-65,-20)\psdot[dotsize=2](-15,-25)\psdot[dotsize=2](-43,4)

\rput(-43,21){$z$}\rput(-58,-6){$d_1$}\rput(-35,-15.5){$d_2$}\rput(-21,0){$d_3$}
\rput(-50,10){$c_1$}\rput(-33.5,3.5){$c_2$}\rput(-30,13){$c_3$}
\rput(-26.5,-7){$a_1$}\rput(-40,-5.5){$a_2$}\rput(-48,-11){$a_3$}
\rput(-69,-20){$b_1'$}\rput(-11,-25){$b_3'$}\rput(-42,7){$b_2'$}

\rput(-60,20){$\mathrm{T}_{13}$}

\end{pspicture}
\end{center} 

The hyperbolic polytope in $\mathbb{B}^6$ associated with $P'$ has symmetry group $S_4$
and the interior points of its quotient by $S_4$ corresponds to the moduli space of maximal 
real quartic curves, that is real smooth quartic curves with $4$ connected components.
For an analysis of the real algebraic geometry of the Kondo period map \cite{Kondo 2000}
we refer to Section 5.7 of the PhD thesis by one of us \cite{Rieken 2015}. The group $S_4$ 
acts as symmetry group of the $4$ connected components of the maximal real quartic curve, 
and hence we find an up to conjugation unique monomorphism 
$S_4\hookrightarrow W(\mathrm{E}_7)/\{\pm1\}$ in the symmetry group 
of the $28$ bitangent of the quartic curve.  The conclusion is that the Weyl groups
$W(\mathrm{E}_6)$ and $W(\mathrm{E}_7)$ have analogous odd presentations as factor groups 
of the hyperbolic Coxeter groups $W(\mathrm{P}_{10})$ and $W(\mathrm{T}_{10})$ modulo
deflation of the free hexagons and octagons respectively, and these presentations have
natural geometric meaning coming from the action of these Weyl groups on the moduli spaces
of marked maximally real del Pezzo surfaces of degree $3$ and $2$ respectively.

\section{Proof of Theorem{\;\ref{projective plane of order 3 theorem}}}

In this section let $n=13$ and let $\mathrm{I}_{26}$ be the incidence graph of
the projective plane $\mathbb{P}^2(3)$ over a field of $3$ elements.
Clearly the chamber $G$ is invariant under the symmetric group $S_{13}$ generated by
the simple roots $s_1,\cdots,s_{12}$ in the notation of Theorem{\;\ref{Vinberg theorem}}.
The connected elliptic subdiagrams (relative to the hyperbolic metric) of $\mathrm{I}_{26}$
are of type $\mathrm{A}_k$ for $k=1,2,3,4$, while the connected parabolic subdiagrams 
are of type $\mathrm{A}_5$ and $\mathrm{D}_4$ of rank $4$ and $3$ respectively.
In order to arrive at the analogue of Lemma{\;\ref{vertices of P over 2 lemma}}
we have the following combinatorial lemma.

\begin{lemma}
The elliptic subdiagrams of $\mathrm{I}_{26}$ of maximal rank $13$
are of type $13\mathrm{A}_1$ or of type
\[ 4\mathrm{A}_1\sqcup3\mathrm{A}_3, \mathrm{A}_1\sqcup4\mathrm{A}_3,
2\mathrm{A}_1\sqcup\mathrm{A}_2\sqcup3\mathrm{A}_3, \mathrm{A}_1\sqcup3\mathrm{A}_4,
\mathrm{A}_2\sqcup\mathrm{A}_3\sqcup2\mathrm{A}_4, 3\mathrm{A}_3\sqcup\mathrm{A}_4 \] 
while the parabolic subdiagrams of maximal rank $12$ are of type $3\mathrm{A}_5$ or 
$4\mathrm{D}_4$. Any two elliptic subdiagrams of $\mathrm{I}_{26}$ of maximal rank 
$13$ and of the same type are conjugated under the automorphism group $\mathrm{PGL}_3(3)\cdot2$.
\end{lemma}

\begin{proof}
It is easy to check that all elliptic subdiagrams of rank $13$ consisting of
only type $\mathrm{A}_1$ factors are those consisting of all black or all white nodes.
Indeed, if there is $1$ white (line) node, then there are at most $9=13-4$ black (point) 
nodes, and so together at most $10$ nodes. If there are $2$ white nodes, then there are 
at most $6=13-7$ black nodes, and so together at most $8$ nodes. If there are $3$ white 
nodes, then there are at most $4=13-9$ black nodes, and so together at most $7$ nodes.
If there are $4$ white nodes, then there are at most $3$ black nodes, and so together
at most $7$ nodes. Hence in all these cases we never arrive at $13$ nodes, and so the 
only cases left of type $13\mathrm{A}_1$ are all black nodes or all white nodes. 

Suppose on the other extreme that we have an elliptic subdiagram of maximal 
rank $13$ with at least one component of type $\mathrm{A}_4$ (any two of these 
are conjugate under the group $\mathrm{PGL_3(3)}$ of colour preserving 
automorphisms of $\mathrm{I}_{26}$). The Coxeter diagram obtained by 
deleting this $\mathrm{A}_4$ together with all nodes connected to this 
$\mathrm{A}_4$ and all bonds connected to these nodes is a free dodecagon 
of type $\tilde{\mathrm{A}}_{11}$. Hence we have to look for an elliptic 
subdiagram of $\tilde{\mathrm{A}}_{11}$ of maximal rank $9$. 
This amounts to finding partitions $n_1+n_2+n_3$ of $12$ into three parts 
$n_i\leq5$, which are just the three possibilities
\[ \{n_i\}=\{2,5,5\},\{3,4,5\},\{4,4,4\} \]
correspondinig to the above types. This leads to the three types of elliptic
subdiagrams of $\mathrm{I}_{26}$ of maximal rank $13$ with at least one component
of type $\mathrm{A}_4$ as given in the lemma.

Now suppose we have an elliptic subdiagram of $\mathrm{I}_{26}$ of maximal rank $13$ 
with at least one factor of type $\mathrm{A}_3$ and the other components of type 
$\mathrm{A}_k$ with $k\leq3$. 
The maximal Coxeter diagram of $\mathrm{I}_{26}$ disconnected from this $\mathrm{A}_3$ 
is a free dodecagon (with nodes numbered $1,\cdots,12$) together with three more nodes, 
each node attached to an opposite pair of odd numbered nodes of the free dodecagon. 
Let us denote this diagram by $\mathrm{I}_{15}$.

The $\mathrm{I}_{15}$ diagram has two essentially different subdiagrams of type
$\mathrm{A}_3$, whose maximal disconnected complement is either a free octagon
of type $\tilde{\mathrm{A}}_7$ or is of type $\tilde{\mathrm{D}}_8$. The diagram
$\tilde{\mathrm{A}}_7$ has no elliptic subdiagrams of rank $7$, but the diagram
$\tilde{\mathrm{D}}_8$ has three elliptic subdiagrams of rank $7$ with only components
of type $\mathrm{A}_k$ for $k\leq3$, namely of type $4\mathrm{A}_1\sqcup\mathrm{A}_3$,
$\mathrm{A}_1\sqcup2\mathrm{A}_3$ and $2\mathrm{A}_1\sqcup\mathrm{A}_2\sqcup\mathrm{A}_3$.

The diagram of type $\mathrm{I}_{15}$ has essentially one subdiagram of type $\mathrm{A}_2$, 
whose maximal disconnected complement is just a free octagon of type $\tilde{\mathrm{A}}_7$
together with two more nodes, one attached to a node of the free octagon and the other
attached to the opposite node of the free octagon. This diagram has just one elliptic 
subdiagram of rank $8$ with only components of type $\mathrm{A}_k$ for $k\leq3$, 
namely the one of type $2\mathrm{A}_1\sqcup2\mathrm{A}_3$.

Since $\mathrm{I}_{15}$ has no elliptic subdiagram of type $10\mathrm{A}_1$ we find that
the $\mathrm{I}_{26}$ diagram has just the three types of elliptic subdiagrams of rank $13$
with at least one component of type $\mathrm{A}_3$ and the others of type $\mathrm{A}_k$
with $k\leq 3$ as given in the lemma.

Now suppose we have an elliptic subdiagram of maximal rank $13$ with at least one factor 
of type $\mathrm{A}_2$. Any two subdiagrams of type $\mathrm{A}_2$ are conjugate under 
the group $\mathrm{PGL_3(3)}$, so we may consider the one with nodes $\{a,f\}$. 
The maximal disconnected subdiagram in $\mathrm{I}_{26}$ has the shape 
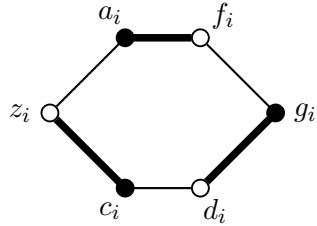
\begin{center}
\psset{unit=1mm}
\begin{pspicture}*(-30,-17)(30,17)

\psline(5,10)(15,0)\psline(-5,-10)(5,-10)\psline(-15,0)(-5,10)
\psline[linewidth=1](-5,10)(5,10)
\psline[linewidth=1](-15,0)(-5,-10)
\psline[linewidth=1](15,0)(5,-10)

\psdot[dotsize=2.5](-5,10)
\psdot[dotsize=2.5](15,0)
\psdot[dotsize=2.5](-5,-10)
\psdot[dotsize=2.5,dotstyle=Bo](5,10)
\psdot[dotsize=2.5,dotstyle=Bo](-15,0)
\psdot[dotsize=2.5,dotstyle=Bo](5,-10)

\rput(-7,13){$a_i$}
\rput(8,13){$f_i$}
\rput(-7,-13){$c_i$}
\rput(7,-13){$d_i$}
\rput(19,0){$g_i$}
\rput(-19,0){$z_i$}

\end{pspicture}
\end{center}
which we denote by $\mathrm{I}_{18}$. This is a trivalent graph with $18$ nodes and 
we need to look for its elliptic subdiagrams of rank $11$. Such an elliptic subdiagram 
is not possible if all connected components of this elliptic subdiagram are of type 
$\mathrm{A}_k$ with $k\leq2$. For example, an elliptic subdiagram of type 
$\mathrm{A}_1\sqcup5\mathrm{A}_2$ is impossible, since we need at least $(3+5.4)/3>7$
additional nodes, which all together makes more than the total available $18$ nodes.
Hence the $\mathrm{I}_{26}$ diagram has no elliptic subdiagram of rank $13$ with
one component of type $\mathrm{A}_2$ and all others of type $\mathrm{A}_k$ with $k\leq2$.

The case of parabolic subdiagrams of $\mathrm{I}_{26}$ of maximal rank $12$ 
is easy, and gives the two possibilities as stated in the lemma.
\end{proof}

The next step is an explicit description of the actual and ideal vertices of the finite
volume hyperbolic polytope obtained from the convex cone
\[ P=\{x\in\mathbb{R}^{13,1};(x,e_p)\leq0,(x,e_l)\leq0\;
\mathrm{for all}\;p\in\mathcal{P},l\in\mathcal{L}\} \]
corresponding to the various types described in the previous lemma. As before we
denote by $e_l=e_0-\sum_{p\in l}e_p$ for $l\in\mathcal{L}$ the simple norm 3 roots.
A projective duality in $\mathrm{PGL}_3(3)\cdot2$ corresponds to an isometry
\[ v_{\mathcal{P}}=e_0\mapsto v_{\mathcal{L}}=(4e_0-\sum_{p\in\mathcal{P}}e_p)/\sqrt3,
e_p\mapsto e_l/\sqrt3 \]
as before. The analogue of Lemma{\;\ref{vertices of P over 2 lemma}} gives the following
description.

\begin{lemma}\label{vertices of P over 3 lemma}
The actual vertices of the hyperbolic polytope obtained from P are represented by the vectors
\[ v_{\mathcal{P}}=e_0,v_{\mathcal{L}}=(4e_0-\sum_{q\in\mathcal{P}}e_q)/\sqrt3 \]
of type $13\mathrm{A}_1$, and by the vectors
\begin{gather*}
v_{pqr}=2e_0-(e_p+e_q+e_r),
v_{lmn}=(5e_0-2\sum_{s\notin l\cup m\cup n}e_s-\sum_{s\in(l\cup m\cup n)'}e_s)/\sqrt3
\end{gather*}
of type $4\mathrm{A}_1\sqcup3\mathrm{A}_3$ for $p,q,r\in\mathcal{P}$ not all on a line and 
$l,m,n\in\mathcal{L}$ not all through a point with $(l\cup m\cup n)'$ the set of $6$ points 
on exactly one of these $3$ lines, and by the vectors 
\[ v_{p,l}=(3e_0-\sum_{q\notin l\sqcup p}e_q),
v_{l,p}=(4e_0-3e_p-\sum_{q\in l}e_q)/\sqrt3 \]
of type $\mathrm{A}_1\sqcup4\mathrm{A}_3$ for all $p\in\mathcal{P},l\in\mathcal{L}$ 
with $p\notin l$, and by the vectors
\begin{gather*}
v_{p,q,l}=3e_0-2e_p-\sum_{r\in l,r\neq q}e_r,
v_{l,m,p}=(7e_0-2\sum_{r\notin l\sqcup p}e_r-\sum_{r\in m,r\neq p}e_r)/\sqrt3
\end{gather*}
of type $2\mathrm{A}_1\sqcup\mathrm{A}_2\sqcup3\mathrm{A}_3$ for all 
$p,q\in\mathcal{P}$ and $l,m\in\mathcal{L}$ with $p\notin l,q\in l$ and likewise
$p\notin l,p\in m$, and by the vectors
\begin{gather*}
v_{p,qrs}=4e_0-2(e_q+e_r+e_s)-(e_u+e_v+e_w) \\
v_{k,lmn}=(7e_0-4e_p-3(e_q+e_r+e_s)-(e_x+e_y+e_z))/\sqrt3
\end{gather*}
of type $\mathrm{A}_1\sqcup3\mathrm{A}_4$ for all $p,q,r,s\in\mathcal{P}$ in general position (no three on a line) with $u\in l_{pq}-l_{rs}-\{p,q\}$, 
$v\in l_{pr}-l_{qs}-\{p,r\}$, $w\in l_{ps}-l_{qr}-\{p,s\}$ for $v_{p,qrs}$ 
and likewise for all $k,l,m,n\in\mathcal{L}$ in general position (no three 
through a point) with $k$ the line through $p$ but not through $q,r,s$ and 
$x=k\cap l_{rs}$, $y=k\cap l_{qs}$, $z=k\cap l_{qr}$ for $v_{k,lmn}$ 
(see picture below), and by the vectors
\begin{gather*}
v_{p,qr,s}=5e_0-3e_p-2(e_q+e_r+e_s)-(e_x+e_y) \\
v_{k,lm,n}=(9e_0-5e_p-4(e_r+e_s)-3n_q-2(n_y+n_z)-n_h)/\sqrt3
\end{gather*}
of type $\mathrm{A}_2\sqcup\mathrm{A}_3\sqcup2\mathrm{A}_4$ for all $p,q,r,s\in\mathcal{P}$ 
in general position for $v_{p,qr,s}$ and likewise for all $k,l,m,n\in\mathcal{L}$ in general 
position with $h=l_{pq}\cap l_{rs}$, $i=l_{pr}\cap l_{qs}$, $j=l_{ps}\cap l_{qr}$
for $v_{k,lm,n}$ (see picture below), and by the vectors
\begin{gather*}
v_{p,q,rs}=3e_0-2e_p-(e_q+e_r+e_s+e_x) \\ 
v_{k,l,mn}=(6e_0-3(e_q+e_r)-2(e_s+e_u+e_v)-(e_w+e_h+e_i))/\sqrt3
\end{gather*}
of type $3\mathrm{A}_3\sqcup\mathrm{A}_4$ for all $p,q,r,s\in\mathcal{P}$ in general position 
for $v_{p,q,rs}$ and likewise for all $k,l,m,n\in\mathcal{L}$ in general position for $v_{k,l,mn}$
(see picture below).

The ideal vertices are represented by the vectors
\[ u_p=e_0-e_p,u_l=(3e_0-\sum_{p\notin l}e_p)/\sqrt3 \]
of type $4\mathrm{D}_4$ for all $p\in\mathcal{P},l\in\mathcal{L}$, and by the vectors
\[ u_{pqrs}=2e_0-(e_p+e_q+e_r+e_s),u_{klmn}=(4e_0-2\sum_pe_p-\sum_qe_q)/\sqrt3 \]
of type $3\mathrm{A}_5$ for any $p,q,r,s\in\mathcal{P}$ in general position and any
$k,l,m,n\in\mathcal{L}$ in general position with $\sum_p$ is the sum over the three 
$p\in\mathcal{P}$ on none of the lines $k,l,m,n$ and $\sum_q$ is the sum over the four 
$q\in\mathcal{P}$ on exactly one of the lines $k,l,m,n$.
\end{lemma}
 
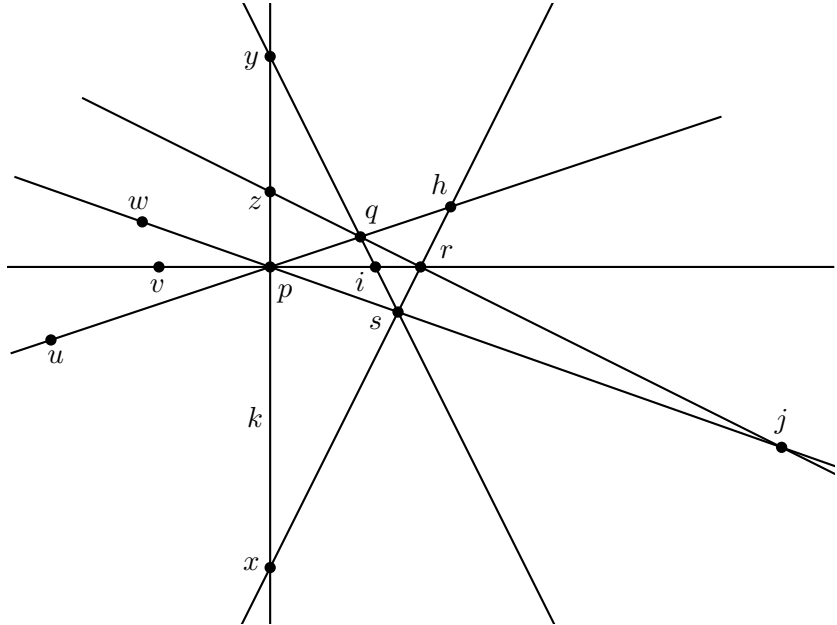
\begin{center}
\psset{unit=0.5mm}
\begin{pspicture}*(-70,-95)(150,70)

\psline(-70,0)(150,0)\psline(0,-100)(0,90)
\psline(-69,-23)(120,40)\psline(-68,24)(204,-72)

\psline(-10,-100)(100,120)\psline(-40,136)(100,-144)\psline(-50,45)(200,-80)

\psdot[dotsize=3](0,0)
\psdot[dotsize=3](24,8)
\psdot[dotsize=3](40,0)
\psdot[dotsize=3](34,-12)
\psdot[dotsize=3](0,-80)
\psdot[dotsize=3](0,56)
\psdot[dotsize=3](0,20)
\psdot[dotsize=3](48,16)
\psdot[dotsize=3](136,-48)
\psdot[dotsize=3](28,0)
\psdot[dotsize=3](-58.285,-19.428)
\psdot[dotsize=3](-29.565,0)
\psdot[dotsize=3](-34,12)

\rput(4,-7){$p$}
\rput(27,14){$q$}
\rput(47,4){$r$}
\rput(28,-15){$s$}
\rput(-5,-79){$x$}
\rput(-5,55){$y$}
\rput(-4,17){$z$}
\rput(-57,-24){$u$}
\rput(-30,-5){$v$}
\rput(-35,17){$w$}
\rput(-4,-40){$k$}
\rput(45,22){$h$}
\rput(24,-4){$i$}
\rput(136,-41){$j$}

\end{pspicture}
\end{center}

\begin{proof}
The vertices $v_{\mathcal{P}},v_{p,l},u_p$ are analogous to the vertices described in
Lemma{\;\ref{vertices of P over 2 lemma}}, and likewise $v_{\mathcal{L}},v_{l,p},u_l$ 
are found by projective duality. It is obvious that $v_{\mathcal{P}}$ is of type
$13\mathrm{A}_1$ as all black nodes in $\mathrm{I}_{26}$. It is also clear from the
picture below that $u_p$ is of type $4\mathrm{D}_4$ (namely $\{l_i,q_i,r_i,s_i\}$ 
for $i=1,2,3,4$) and that $v_{l,p}$ is of type $\mathrm{A}_1\sqcup4\mathrm{A}_3$
(namely $\{l\}\sqcup\{l_i,q_i,r_i\}$ for $i=1,2,3,4$).
\begin{center}
\psset{unit=0.5mm}
\begin{pspicture}*(-30,-40)(120,70)

\psline(80,-40)(80,60)
\psline(-30,-15)(100,50)
\psline(-30,-7.5)(100,25)
\psline(-30,0)(100,0)
\psline(-30,7.55)(100,-25)

\psdot[dotsize=3](0,0)
\psdot[dotsize=3](30,15)
\psdot[dotsize=3](60,30)
\psdot[dotsize=3](80,40)
\psdot[dotsize=3](32,8)
\psdot[dotsize=3](64,16)
\psdot[dotsize=3](80,20)
\psdot[dotsize=3](32,0)
\psdot[dotsize=3](64,0)
\psdot[dotsize=3](80,0)
\psdot[dotsize=3](30,-7.5)
\psdot[dotsize=3](60,-15)
\psdot[dotsize=3](80,-20)

\rput(0,-6){$p$}
\rput(30,20){$q_1$}
\rput(58,35){$r_1$}
\rput(85,36){$s_1$}
\rput(37,4){$q_2$}
\rput(63,10){$r_2$}
\rput(85,16){$s_2$}
\rput(37,-4){$q_3$}
\rput(63,-5){$r_3$}
\rput(85,-4){$s_3$}
\rput(28,-12){$q_4$}
\rput(58,-20){$r_4$}
\rput(85,-26){$s_4$}
\rput(76,60){$l$}
\rput(105,50){$l_1$}
\rput(105,25){$l_2$}
\rput(105,0){$l_3$}
\rput(105,-25){$l_4$}

\end{pspicture}
\end{center}

We shall describe the details for finding the ideal vertices $u_{pqrs}$ of type
$3\mathrm{A}_5$. Consider the maximal tree subdiagram of $\mathrm{I}_{26}$ of type 
$\mathrm{Y}_{555}$ with nodes $\{f,e_i,d_i,c_i,b_i,a_i\}$ for $i=1,2,3$. 
There are $4$ remaining black nodes in the $\mathrm{I}_{26}$ diagram, namely 
$p=a,q=g_1,r=g_2,s=g_3$. It is now obvious from the $\mathrm{I}_{26}$ diagram that 
$(u_{pqrs},e_x)=0$ for all $9$ black (point) nodes $x$ of the $\mathrm{Y}_{555}$ 
subdiagram, and also 
\[ (u_{pqrs},e_y)=(u_{pqrs},e_0-\sum_{x\in y}e_x)=-2+|\{p,q,r,s\}\cap y|=0 \]
for all $6$ white (line) nodes $y\neq f$ of the $\mathrm{Y}_{555}$ subdiagram. The other
ideal vertices $u_{klmn}$ are obtained by projective duality.

Similar straightforward calculations work for the remaining actual vertices, 
for example using that the elliptic subdiagrams
\[ \mathrm{A}_1\sqcup3\mathrm{A}_4, \mathrm{A}_2\sqcup\mathrm{A}_3\sqcup2\mathrm{A}_4,   
3\mathrm{A}_3\sqcup\mathrm{A}_4 \] 
of $\mathrm{I}_{26}$ are all realized within a $\mathrm{Y}_{555}$ subdiagram.
\end{proof}

In order to show that the chamber $P$ for the diagram $\mathrm{I}_{26}$ is
contained in the chamber $G$ for the norm one roots we check that all the vertices 
of $P$ are contained in $G$. Since $G$ is invariant under the symmetric group $S_{13}=W(\mathrm{A}_{12})$ (with $\mathrm{A}_{12}$ the subdiagram of the second 
Coxeter--Vinberg diagram in Theorem{\;\ref{Vinberg theorem} with nodes numbered 
$1,\cdots,12$) this amounts to checking that all conjugates under $S_{13}$ of the 
$12$ actual and $4$ ideal vertices enumerated in the previous lemma lie in $G$.

\begin{lemma}
All vertices of $P$ given in the previous lemma are contained in the fundamental 
chamber $G$ for the norm $1$ roots.
\end{lemma}

\begin{proof}
This is a straightforward computation. For example, the last ideal vertex
$u:=4e_0-2(e_1+e_2+e_3)-(e_4+e_5+e_6+e_7)$, which will be abbreviated $42^31^4$, satisfies $(u,\alpha_0=e_0-e_1-e_2-e_3)=-4+6=2$. Hence $s_0(u)=u-2\alpha_0$ is conjugated under $S_{13}$ to $u:=21^4$. Since $(u,\alpha_0)=-2+3=1$ we see that 
$s_0(u)=u-\alpha_0$ is conjugated under $S_{13}$ to $u:=11$, which stands for 
$e_0-e_1$. This is a cusp of $D$. Hence our original vertex $u$ of $P$ is 
separated from $D$ by mirrors in norm $2$ roots, and therefore $u$ lies in $G$. 

Using our compact notation the $18$ cases can be conveniently summarized in the 
following table and the lemma follows from the observation that in each row the most 
right symbol represents a point in the chamber $D$ for $\Gamma^{13}$.

\begin{center}
\begin{tabular}{|l|l|l|l|l|} \hline
$13\mathrm{A}_1$ & $1$ & $ $ & $ $ & $ $  \\ 
$ $ & $41^{13}$ & $ $ & $ $ & $ $  \\ 
$4\mathrm{A}_1\sqcup3\mathrm{A}_3$ & $21^3$ & $1$ & $ $ & $ $  \\ 
$ $ & $52^41^6$ & $421^9$ & $ $ & $ $  \\ 
$\mathrm{A}_1\sqcup4\mathrm{A}_3$ & $31^8$ & $ $ & $ $ & $ $  \\ 
$ $ & $431^4$ & $321^2$ & $21$ & $ $  \\ 
$2\mathrm{A}_1\sqcup\mathrm{A}_2\sqcup3\mathrm{A}_3$ & $321^3$ & $21^2$ & $ $ & $ $ \\
$ $ & $73^22^61$ & $62^71^2$ & $ $ & $ $ \\
$\mathrm{A}_1\sqcup3\mathrm{A}_4$ & $42^31^3$ & $21^3$ & $1$ & $ $  \\ 
$ $ & $743^31^3$ & $431^4$ & $321^2$ & $21$  \\ 
$\mathrm{A}_2\sqcup\mathrm{A}_3\sqcup2\mathrm{A}_4$ & $532^31^2$ & $321^3$ & $21^2$ & $ $  \\ 
$ $ & $954^232^21$ & $532^21^2$ & $31^3$ & $ $  \\ 
$3\mathrm{A}_3\sqcup\mathrm{A}_4$ & $321^4$ & $21^3$ & $1$ & $ $  \\
$ $ & $63^22^31^3$ & $42^21^5$ & $31^6$ & $ $  \\ 
\hline
$4\mathrm{D}_4$ & $11$ & $ $ & $ $ & $ $  \\ 
$$ & $31^9$ & $ $ & $ $ & $ $  \\ 
$3\mathrm{A}_5$ & $21^4$ & $11$ & $ $ & $ $  \\ 
$ $ & $42^31^4$ & $21^4$ & $11$ & $$  \\ \hline
\end{tabular}
\end{center}

Since we only used reflections $s_i$ for $i=0,1,\cdots,12$ in our reduction pattern
we have in fact shown that
\[ P\subset \cup_{w\in W(\mathrm{E}_{13})}wD\subset G \] 
with $\mathrm{E}_{13}$ the subdiagram of the second Coxeter--Vinberg diagram
in Theorem{\;\ref{Vinberg theorem}} with nodes numbered $0,1,\cdots,12$.

\end{proof}

This finishes the proof of Theorem{\;{\ref{projective plane of order 3 theorem}}}

\section{Conclusions}

The nodes of the $\mathrm{I}_{26}$ diagram are indexed by the set 
$\mathcal{I}=\mathcal{P}\sqcup\mathcal{L}$ of points and lines in $\mathbb{P}^2(3)$.
Let us write $\omega=(-1+\sqrt{-3})/2$ and $\theta=\sqrt{-3}$.
The Allcock lattice $L$ is a Lorentzian lattice over the ring of Eisenstein integers
$\mathcal{E}=\mathbb{Z}+\mathbb{Z}\omega$ with generators $\varepsilon_i$ for 
$i\in\mathcal{I}$ and Hermitian inner product
\[ \langle\varepsilon_i,\varepsilon_i\rangle=3,
   \langle\varepsilon_j,\varepsilon_k\rangle=0,
   \langle\varepsilon_p,\varepsilon_l\rangle=\theta\]
for all $i,j,k\in\mathcal{I}$ with $j\neq k$ disconnected and  all 
$p\in\mathcal{P}$ and $l\in\mathcal{L}$ connected in $\mathrm{I}_{26}$. 
Note that $\langle\lambda,\mu\rangle\in\mathcal{E}\theta$ for all $\lambda,\mu\in L$.
It is easy to see that $L$ has rank $14$ and signature $(13,1)$ and so is Lorentzian.

A vector $\varepsilon\in L$ of norm $\langle\varepsilon,\varepsilon\rangle=3$ is called 
a root. The order three complex reflection
\[ t_{\varepsilon}(\lambda)=\lambda+(\omega-1)\frac{\langle\lambda,\varepsilon\rangle}
   {\langle\varepsilon,\varepsilon\rangle}\varepsilon \]
is a unitary automorphism  of $L$, and is called the triflection with root $\varepsilon$. 
It was shown by Allcock that the roots in $L$ form a single orbit under $\mathrm{U}(L)$,
and Basak proved that $\mathrm{U}(L)$ is generated by triflections 
\cite{Allcock 2000},\cite{Basak 2006}. 

We extend scalars from $\mathcal{E}$ to $\mathbb{Z}[\zeta_{12}]$ and put 
\[  e_p=\varepsilon_p,e_l=-\sqrt{-1}\varepsilon_l \]
for $p\in\mathcal{P}$ and $l\in\mathcal{L}$. The Gram matrix of the $\{e_i\}$ becomes
\[ \langle e_i,e_i\rangle=3,\langle e_i,e_j\rangle=0,\langle e_j,e_k\rangle=-\sqrt3 \]
for all $i,j,k\in\mathcal{I}$ with $i\neq j$ disconnected and $j\neq k$ connected.
The real vector space $V$ spanned by the vectors $\{e_i;i\in\mathcal{I}\}$ is a 
Lorentzian vector space of dimension $14$. 

\begin{lemma}
The intersection $L\cap V$ is an integral Lorentzian lattice with inner product values
all contained in $3\mathbb{Z}$, and the integral lattice $L_r=L\cap V$ with the scalar 
product $(\cdot,\cdot)=\langle\cdot,\cdot\rangle/3$ is isomorphic to the odd unimodular 
lattice $\mathbb{Z}^{13,1}$. 
\end{lemma}

\begin{proof}
Indeed, inner product values of $L\cap V$ lie in $\mathcal{E}\theta\cap\mathbb{R}=3\mathbb{Z}$.
Hence for a root $\varepsilon\in L$ either $e=\omega^j\varepsilon\in L_r$ for some $j$ has norm $1$ 
(for example for $\varepsilon=\varepsilon_p$ for some $p\in\mathcal{P}$) or 
$\omega^j\varepsilon\notin L_r$ for all $j$ and $e=\omega^j\theta\varepsilon\in L_r$ for some $j$ has norm $3$
(for example for $\varepsilon=\varepsilon_l$ for some $l\in\mathcal{L}$).
Since the discriminant of the Eisenstein lattice $L$ is equal to $3^7$ it follows that 
the discriminant of the integral lattice $L_r$ should divide $3^{14}/3^{14}=1$.
Hence the lattice $L_r$ is unimodular, odd, of signature $(13,1)$ and so
isomorphic to $\mathbb{Z}^{13,1}$ \cite{Serre 1970}.
\end{proof}

The conclusion is that the intersection of the complex mirror arrangement in
$\mathbb{C}\otimes_{\mathbb{R}}V$ for all norm three roots in $L$ with the real 
form $V$ consists of the mirror arrangement for all norm one roots in $L_r$ together 
with the transform of this arrangement under the orthogonal involution of $V$ 
coming from a projective duality on $\mathbb{P}^2(3)$. Hence the 
results of the previous section indeed prove Conjecture 1.7 made by one of us 
in \cite{Heckman 2013}. Using the results of that paper we arrive at 

\begin{theorem}\label{monstrous proposal theorem}
Let $\mathbb{B}=\{z\in\mathbb{C}\otimes_{\mathbb{R}}V;\langle z,z\rangle<0\}/\mathbb{C}^{\times}
\subset\mathbb{P}(\mathbb{C}\otimes_{\mathbb{R}}V)$ 
be the complex hyperbolic ball and let $\mathbb{B}^{\circ}$ be the mirror
arrangement complement. If $\Gamma=\mathrm{PU}(L)$ then the orbifold fundamental group 
$\Pi_1^{\mathrm{orb}}(\mathbb{B}^{\circ}/\Gamma)$, which according to a theorem of 
Allcock and Basak \cite{Allcock--Basak 2014} is isomorphic to a factor group of the Artin group 
$\mathrm{Art}(\mathrm{I}_{26})$ with generators $T_i$ for $i\in\mathcal{I}$ and braid relations
\[ T_iT_j=T_jT_i\;\;,\;\;T_kT_lT_k=T_lT_kT_l \]
for all $i,j,k,l\in\mathcal{I}$ with $i,j$ disconnected and $k,l$ connected, has as 
factor group after imposing the Coxeter relations $T_i^2=1$ one of the following three groups, 
either $M\wr2=(M\times M)\rtimes S_2$ (the bimonster group) or $S_2$ or the trivial group. 
\end{theorem} 

This provides a partial positive answer to the monstrous proposal of Daniel Allcock 
\cite{Allcock 2009}.

\noindent
Gert Heckman, Radboud University Nijmegen: g.heckman@math.ru.nl
\newline
Sander Rieken, Radboud University Nijmegen: s.rieken@math.ru.nl


\begin{thebibliography}{25}

\bibitem{Allcock 2000}
Daniel Allcock, The Leech lattice and complex hyperbolic reflections,
Invent. Math. {\bf 140} (2000), 283-301.

\bibitem{Allcock 2009}
Daniel Allcock, A Monstrous Proposal, Groups and symmetries, CRM Proc. 
Lecture Notes {\bf 47}, Amer. Math. Soc., Providence, RI. (2009) 17-24 

\bibitem{Allcock--Basak 2014}
Daniel Allcock and Tathagata Basak, Geometric generators for braid-like groups,
arXiv:math.GT/1403.2401, 10 March 2014.

\bibitem{Allcock--Carlson--Toledo 2002}
Daniel Allcock, Jim Carlson and Domingo Toledo, The complex hyperbolic geometry
of the moduli space of cubic surfaces, J. Algebraic Geometry {\bf 11} (2002), 659-724.

\bibitem{Allcock--Carlson--Toledo 2010}
Daniel Allcock, Jim Carlson and Domingo Toledo, Hyperbolic geometry and moduli
of real cubic surfaces, Ann. Sci. \'{E}cole Norm. Sup. {\bf 43} (2010), 69-115. 

\bibitem{Basak 2006}
Tathagata Basak, The complex Lorentzian Leech lattice and the bimonster I,
J. Alg.{\bf 309} (2007), 32-56.

\bibitem{Basak 2007}
Tathagata Basak, Reflection group of the quaternionic Lorentzian Leech lattice, 
J. Alg. {\bf 309} (2007), 57-68.

\bibitem{Basak 2012}
Tathagata Basak, The complex Lorentzian Leech lattice and the bimonster II,
arXiv:math.GR/0811.0062, 6 April 2012. 

\bibitem{ATLAS 1985}
J.H. Conway, R.T. Curtis, S.P. Norton, R.A. Parker, R.A. Wilson, The ATLAS of finite
groups, Oxford University Press, 1985.

\bibitem{Chu 2011}
Kenneth C.K. Chu, On the Geometry of the Moduli Space of Real Binary Octics,
Canad. J. Math. {\bf 63} (2011), 755-797.

\bibitem{Conway--Simons 2001}
John H. Conway and Christopher Simons, 26 Implies the Bimonster,
Journal of Algebra {\bf 235} (2001), 805-814.

\bibitem{Dolgachev--Ortland}
I.V. Dolgachev and D. Ortland, Points Sets in Projective Spaces and Theta Functions,
Ast\'{e}risque {\bf 165}, 1988.

\bibitem{Everitt--Ratcliffe--Tschantz 2010}
Brent Everitt, John G. Ratcliffe and Steven T. Tschantz, Right-angled Coxeter polytopes, 
hyperbolic 6-manifolds and a problem of Siegel, Math. Ann. {\bf 354} (2012), 871-905. 

\bibitem{Heckman 2013}
Gert Heckman, The Allcock Ball Quotient, arXiv:math.AG/1307.1339, 4 July 2013.

\bibitem{Heckman CG 2013}
Gert Heckman, Coxeter Groups, Informal Lecture Notes, Fall 2013.

\bibitem{Heckman--Rieken 2014a}
Gert Heckman and Sander Rieken, An Odd Presentation of $W(\mathrm{E}_6)$, 
Preprint 2014.

\bibitem{Heckman--Rieken 2015}
Gert Heckman and Sander Rieken, Hyperbolic Geometry and Moduli of Real Curves of Genus Three,
to appear.

\bibitem{Ivanov 1992} A.A. Ivanov, A geometric characterization of the monster,
Durham Conference 1990, London Math. Soc. Lecture Notes Ser. {\bf 165},
Cambridge University Press (1992), 46-62.

\bibitem{Kondo 2000}
Shigeyuki Kondo, A complex hyperbolic structure for the moduli space of curves
of genus three, Journal reine angewandte Mathematik {\bf 525} (2000), 219-232.

\bibitem{Norton 1992}
S.P. Norton, Constructing the monster, Durham Conference 1990, 
London Math. Soc. Lecture Notes Ser. {\bf 165},
Cambridge University Press (1992), 63-76.

\bibitem{Rieken 2015}
Sander Rieken, Moduli of real curves of genus three, PhD Radboud University, 12 February 2015.

\bibitem{Serre 1970}
J.-P. Serre, Cours d' Arithm\'{e}tique, Presses universitaires de France, 1970.

\bibitem{Simons 2001}
Christopher S. Simons, Deflating infinite Coxeter groups to finite groups,
Prceedings on Moonshine and Related Topics, CRM Proceedings Lecture Notes {\bf 30},
American Mathematical Society, Providence (2001), 223-229.

\bibitem{Simons 2005}
Christopher S. Simons, An elementary Approach to the Monster,
The American Mathematical Monthly {\bf 112} (2005), 334-341.

\bibitem{Vinberg 1980}
E.B. Vinberg, Hyperbolic reflection groups, Russian Math. Surv. {\bf 40} (1980), 31-75.

\bibitem{Yoshida 2001}
Masaaki Yoshida, A hyperbolic structure onn the real locus of the moduli space
of marked cubic surfaces, Topology {\bf 40} (2001), 469-473.

\end{thebibliography}
\end{document}